\documentclass[12pt,a4paper]{amsart}
\usepackage{amsmath}
\usepackage{amssymb}
\usepackage{latexsym}
\usepackage{amsmath}
\usepackage{amsbsy}
\usepackage{amssymb}
\usepackage{amscd}
\usepackage{comment}
\usepackage{enumitem}
\usepackage{xcolor}

\DeclareMathOperator{\Fl}{Fl}

\providecommand{\abso}[1]{\left\lvert#1\right\rvert}

\newcommand{\coleq}{\mathrel{\mathop:}=}

\newtheorem{theorem}{\bf Theorem}
\newtheorem{definition}[theorem]{\bf Definition}

\newtheorem{corollary}[theorem]{\bf Corollary}
\newtheorem{proposition}[theorem]{\bf Proposition}
\newtheorem{remark}[theorem]{\bf Remark}
\newtheorem{lemma}[theorem]{\bf Lemma}

\newcommand{\e}{\varepsilon}

\def\sk{\omega}
\def\SK{{\mathrm{SK}}}
\def\TO{{\mathrm{TO}}}

\newtheorem{thr}{\hspace*{-1.1mm}}[section]
\newcommand{\bt}{\begin{thr} {\bf Theorem }}
\newcommand{\et}{\end{thr}}
\newcommand{\bp}{\begin{thr} {\bf Proposition }}
\newcommand{\bc}{\begin{thr} {\bf Corollary }}
\newcommand{\blem}{\begin{thr} {\bf Lemma }}
\newcommand{\bex}{\begin{thr} {\bf Example }\rm}
\newcommand{\bexs}{\begin{thr} {\bf Examples }\rm}
\newcommand{\bd}{\begin{thr} {\bf Definition }}
\newcommand{\br}{\begin{thr} {\bf Remark}}

\def\sup{\mathop{\textrm sup}\limits}

\def\MC{\tilde{\mathcal{A}}}  

\def\SK{{\rm SK}}
\newcommand{\cT}{\mathcal{T}}

\def\st#1#2{{}^{#1}_{#2}}         
\def\gt#1#2#3#4{\st{#1\ldots#2}{#3\ldots#4}}

\def\half{{\scriptstyle{1 \over 2}}}

\let\eps=\varepsilon              
\let\le=\leqslant

\def\Real{{\Bbb R}}               
\def\Nat{{\Bbb N}}                
\def\tR{{\tilde R}}

\def\cD{{\mathcal{D}}}               
\def\tCD{{\tilde{\mathcal{D}}}}      
\def\TM{{\frak X}}              
\def\tensor{{\mathcal{T}}}           
\def\CE{{\mathcal{E}}}               
\newcommand{\atil}{\ensuremath{\tilde{\mathcal{A}}_0(M)} } 
\def\CEM{\hat \CE}              
\def\CM{{{\hat{\mathcal{E}}_M}}}       
\def\CN{\hat{\mathcal{N}}}               
\def\CG{\hat{\mathcal{G}}}          
\def\hCG{\check{\mathcal{G}}}          
\def\delo{{\mathop{\nabla}\limits^0}} 
\def\Lie{{\mathcal{L}}}              
\def\gLie{\hat \Lie}        
\def\dLie{\Lie}             

\def\transp{\Upsilon}           
\def\transpp{\Xi}           

\def\emph#1{{\it#1}}

\def\k{{\boldsymbol k}}
\def\l{{\boldsymbol l}}

\newcommand{\Ll}{L_{\mbox{\rm\small loc}}}
\newcommand{\Hl}{H_{\mbox{\rm\small loc}}}

\def\form{\Omega^1(M)}
\def\vect{\TM}
\def\Riem{{\rm Riem}}

\begin{document}

\title[nonlinear distributional geometry]
{A nonlinear theory of distributional geometry}

\author{E.~A.~Nigsch}
\address{E.~A.~Nigsch, Institut f\" ur Mathematik, Universit\" at Wien, Vienna, Austria}
\email{eduard.nigsch@univie.ac.at}

\author{J.~A.~Vickers}
\address{J.~A.~Vickers, School of Mathematics, University of Southampton, Southampton SO17 1BJ, UK}
\email{J.A.Vickers@soton.ac.uk}



\label{firstpage}

\begin{abstract}
  This paper builds on the theory of generalised functions begun in
  \cite{paper1}. The Colombeau theory of generalised scalar fields on
  manifolds is extended to a nonlinear theory of generalised tensor
  fields which is diffeomorphism invariant and has the sheaf property. The generalised Lie derivative for generalised tensor fields
  is introduced and it is shown that this commutes with the embedding
  of distributional tensor fields. It is also shown that the covariant
  derivative of generalised tensor fields commutes with
  the embedding at the level of association. The concept of
  generalised metric is introduced and used to
  develop a nonsmooth theory of differential geometry.  It is shown
  that the embedding of a continuous metric results in a generalised metric
  with well defined connection and curvature.  It is also shown that a
  twice continuously differentiable metric which is a solution of the vacuum Einstein equations
  may be embedded into the algebra of generalised tensor fields and has generalised
  Ricci curvature associated to zero. Thus, the embedding preserves
  the Einstein equations at the level of association. Finally, we
  consider an example of a metric which lies outside the
  Geroch-Traschen class and show that in our diffeomorphism invariant
  theory the curvature of a cone is associated to a delta function.
\end{abstract}

\maketitle

\section{Introduction}

In a previous paper \cite{paper1} we introduced a global theory of
generalised functions on a manifold $M$. The key idea was to replace a
nonsmooth function $f$ by 1-parameter families of
smooth functions according to 
\begin{equation}\label{alpha} \tilde
    f_\e(x)=\int_Mf(y)\sk_{x, \e}(y) ,
  \end{equation}

depending on a suitable family of smoothing kernels $(\sk_\e)_\e$.
For fixed $\e$ these may be treated just like smooth functions on manifolds
so all the standard operations that may be carried out on smooth
functions extend to the smoothed functions $\tilde f_\e$. The embedding \eqref{alpha}
extends to distributions $T \in \cD'(M)$ by defining
\begin{equation} \label{embed}
\tilde T(\sk_\e)(x) = \langle T, \sk_{x,\e} \rangle.
\end{equation}
By introducing certain asymptotic conditions on the basic space of
generalised functions one may define the spaces of moderate and negligible
generalised functions and hence obtain the Colombeau algebra of
generalised functions on $M$ as the quotient
$\CG(M)=\CM(M)/\CN(M)$. The algebra of generalised functions $\CG(M)$
contains the space of smooth functions as a subalgebra and the space of
distributions as a canonically embedded linear subspace. We also
introduced both the generalised Lie derivative and the covariant
derivative of generalised scalar fields on $M$. The generalised Lie
derivative commutes with the embedding while the covariant derivative
commutes at the level of association.

For applications of these ideas to general relativity we are
interested in looking at Einstein's equations for metrics of low
differentiability, which are tensorial rather than scalar
objects. Because the embedding into the algebra does not commute with
multiplication (except in the smooth case) one cannot simply work with
the coordinate components of a tensor and use the theory of
generalised scalars.

In section 2 we show how it is possible to define an algebra of
generalised tensor fields on a manifold which contains the algebra of smooth
tensor fields as a subalgebra and has a canonical coordinate
independent embedding of the spaces of $(r,s)$-tensor distributions as
linear subspaces. In section 3 we look at the
embedding of distributional tensor fields into the algebra of
generalised tensor fields.  In section 4 the generalised
Lie derivative is introduced and it is shown that it commutes with the
embedding.  In section 5 we use the theory described earlier to
develop a nonlinear theory of distributional geometry and briefly look
at applications to general relativity in section 6. The covariant
derivative of a generalised tensor field is introduced and it is shown
that this commutes with the canonical embedding at the level of
association. We then consider generalised metrics and show that the
embedding of a $C^0$ metric results in a generalised metric with well
defined connection and curvature. We also show that if one embeds a
$C^2$ vacuum metric into the Colombeau algebra then its generalised
Ricci curvature vanishes at the level of
association. Finally we look at an example of a metric for which it is
not possible to define the curvature using conventional distributions
and show that the generalised Einstein tensor of a cone is associated
to a distributional energy momentum tensor in a canonical and
coordinate independent manner.

We will continue to use the notation of \cite{paper1}. In particular, $\mathfrak{X}$ and $\Omega^p(M)$ denote the spaces of smooth vector fields and $p$-forms, respectively. A distributional tensor field may be regarded as simply a tensor field with distributional coefficients. However we prefer to follow \cite{marsden} and adopt a global description in which type $(r,s)$ tensor fields are regarded as dual to type $(s,r)$ tensor densities. We denote the space of compactly supported type $(s,r)$ tensor densities $\tilde\cD^s_r(M)$ and denote the space of type $(r,s)$ tensor distributions $\cD'^r_s(M)$. We let $\tilde\cD(M) = \tilde \cD^0_0(M)$ denote the space of densities. Note that on an orientable manifold scalar densities are equivalent to $n$-forms so in the scalar case what we do here is consistent with \cite{paper1}.

\section{The algebra of generalised tensor fields}
In this section we will extend the theory of generalised scalar
fields on a manifold $M$ presented in \cite{paper1} to vectors,
covectors and more general tensor fields. Before giving the precise
definitions we motivate these by looking at the smoothing of
continuous (or more generally locally integrable) tensor fields on $M$
by integration.

Given a scalar field $f \in C^0(M)$ we may define a smooth scalar field $\tilde f_\e$ by \eqref{alpha}.
Unfortunately this does not make sense if we replace $f$ by a vector field $X$. One obvious possibility is to work in some local coordinate system and define (leaving out the $\e$ for the moment)
\[ \tilde X^a(x)=\int_M X^a(y) \sk_{x}(y). \]
However, if we transform to a new coordinate system $x'$ and then smooth we find
\[
{\tilde X}^{a'}(x)=\int_M {{\partial x^{a'}} \over {\partial x^b}}(y)X^b(y) \sk_x(y)
\]
which in general is not equal to
\[
{{\partial {x^a}'} \over {\partial x^b}}(x){\tilde X}^b(x)={{\partial {x^a}'} \over {\partial x^b}}(x)\int_MX^b(y) \sk_x(y).
\]
The reason for the problem is that we are attempting to integrate the components of a vector at different
points (see \cite{VW2} for details). To make such an integral well
defined in a coordinate invariant way we need to prescribe some additional structure which
enables us to compare tangent spaces at different points of the
manifold.

Let $\transp(x,y) \in T_xM \otimes T_y^*M$ be a two point tensor that 
depends smoothly on $x$ and $y$. More precisely, $\transp$ is an element of $\TO(M) \coleq \Gamma(M \times M, TM
\boxtimes T^*M)$ and will be called \emph{transport operator} (see \cite{advances2} and \cite{distcurv} for
details).

For $x,y \in M$, $\transp$ defines a map
\[ \transp^*(x,y) \colon T_x^*M \to T_y^*M \]
which may be written using the abstract index convention of \cite{penrose} as
\[ \omega_a \mapsto {\transp}^a{}_b(x,y) \omega_a.  \]
Contracting $\transp^a{}_b(x,y)$ with a covector $\omega_a$ in $T_x^*M$
hence gives an element of $T_y^*M$. 
We may also use $\transp$ to define a map
\[ \transp_*(x,y) \colon T_xM \to T_yM \]
by the assignment
\[ X^a \mapsto {\transp}^a{}_b(y,x)X^b = \transp_b{}^a(x,y)X^b \]
where we set $\transp_b{}^a(x,y) \coleq \transp^a{}_b(y,x)$.

By taking suitable tensor products $\transp$ may be used to transport arbitrary tensors from $x$ to $y$.
Note that the transport operators we will use in the development of the theory typically satisfy $\transp^a_b = \delta^a_b$, hence $\transp^*$ and $\transp_*$ are the identity on the diagonal and invertible in a neighborhood of it.

We are now in a position to describe the smoothing of a locally integrable vector field $X$. Let $x \in M$, $\transp$ be a transport operator and let $\sk \in \SK(M)$ be a smoothing kernel; then we define $\tilde X(x)$
by its action on covectors $\alpha \in T_x^*M$ (which may be written using
the abstract index convention) as
\[ \tilde X^a(x)\alpha_a=\int_{y \in M} \alpha_a{\transp}^a{}_b(x,y)X^b(y)\sk_x(y). \]
Note that $\alpha_a{\transp}^a{}_b(x,\cdot)X^b(\cdot)$ is a scalar field on $M$ which may be smoothed by integrating against $\sk_x$.

Similarly, in order to smooth a locally integrable covector field $\beta$ we consider its action on vectors $Y \in T_xM$ and use the transport operator to extend this to a vector field. Thus,
\[ \tilde \beta_a(x)Y^a=\int_{y \in M} Y^a{\transp}_a{}^b(x,y)\beta_b(y)\sk_x(y). \]

Using the same strategy we can smooth a general locally integrable type $(r,s)$ tensor field $S$ by defining $\tilde S$ according to
\begin{equation}
\begin{aligned}
\tilde S\gt{a_1}{a_r}{b_1}{b_s}(x)&=
\int_M S\gt{c_1}{c_r}{d_1}{d_s}(y){\transp}^{a_1}{}_{c_1}(x,y)\dots \\
&\qquad {\transp}^{a_r}{}_{c_r}(x,y){\transp}_{b_1}{}^{d_1}(x,y)\dots
{\transp}_{b_s}{}^{d_s}(x,y) \sk_x(y)
\label{46}
\end{aligned}
\end{equation}
which with some changes of notation is the formula given in \cite{VW2}.

A natural way of obtaining such transport operators is by using a
background connection $\gamma$. If we choose $U$ to be some
geodesically convex neighbourhood for $\gamma$ (i.e., an open set such
that every pair of points in $U$ can be connected by a unique geodesic
lying in the set) then we may define $\transp_*(x,y)$ to be given by
parallel transport of vectors along the geodesic connecting $x$ to
$y$. Note that for such a transport operator for $x \in U$,
${\transp}_a{}^b(x,x)=\delta^a_b$ (which ensures that $\tilde X^a(x)
\to X^a(x)$ as $\eps \to 0$). Unfortunately, such a transport operator
is only defined for $(x,y) \in U \times U$. However, using a partition
of unity we may define a global transport operator which is determined
by $\gamma$ in the above way in a neighborhood of the diagonal.

We are now in a position to define the basic tensor space that we will use to define generalised tensor field on manifolds. 

\begin{definition}[Basic tensor space] \label{basictensors}
The basic space $\CEM^r_s(M)$ of type $(r,s)$ generalised tensor fields consists of all maps
\[ S \colon \TO(M) \times \SK(M) \to \cT^r_s(M) \]
such that $S(\transp, \sk)$ depends smoothly on $\Upsilon \in \TO(M)$ and $\sk \in \SK(M)$, where $\SK(M)=C^\infty(M,\tilde\cD(M))$ is the space of smoothing kernels \cite[Definition 13]{paper1} and the smoothness with respect to $\omega$ and $\Upsilon$ is defined using the definitions of \cite{KM}.
\end{definition}

Note that as in our previous paper, for the sake of presentation we completely omit discussion of the sheaf property; to obtain it we actually would have to restrict the basic space to a somewhat smaller one. For details, we refer to \cite{distcurv,specfull}.

Before going on to define moderate and negligible generalised tensor fields we will look at the properties of the basic space $\CEM^r_s(M)$.

\section{Embedding distributional tensor fields}

In this section we will discuss the embedding of distributional tensor fields into the space of generalised tensor fields. We have already given the basic construction for the embedding of a continuous tensor field $S$ in equation \eqref{46}. We now turn to the case of a distributional tensor field $T$. 

Given a type $(r,s)$ distributional tensor field $T \in \cD'^r_s(M)$, a smoothing kernel $\sk_x(y)$ and  a transport operator $\transp \in \TO(M)$ we may define a smooth tensor field $\tilde T(\transp, \sk) \in \cT^r_s(M)$ according to
\begin{equation}
\tilde T(\transp, \sk)(x)( \alpha^1\dots\alpha^r,Y_1\dots Y_s) =\langle T, \Psi_x \rangle \label{49}
\end{equation}
where $\alpha^j \in T^*_xM$ for $j=1 \dots r$, $Y_k \in T_xM$ for $k=1 \dots s$ and $\Psi_x(y)$ is the type $(s,r)$ tensor density given by
\begin{equation}\label{beta}
\begin{aligned}
\Psi_x(y) = & (\transp^*(x,y)\alpha^1(x))\otimes\dots \otimes(\transp^*(x,y)\alpha^r(x))\otimes \\
& (\transp_*(x,y)Y_1(x))\otimes\dots \otimes(\transp_*(x,y)Y_s(x))\otimes\sk_x(y).
\end{aligned}
\end{equation}
The above formula (\ref{49}) therefore gives a canonical embedding
\begin{align*}
\iota^r_s \colon \cD'^r_s(M) &\to \CEM^r_s(M), \\
T &\mapsto \tilde T.
\end{align*}
It can be shown that $\tilde T(\transp, \sk)$ depends smoothly on the smoothing kernel $\sk$ and on the choice of transport operator $\transp$. Actually, the seemingly innocuous statement that this mapping is smooth is far from trivial to prove and is considered in detail in \cite{distcurv,vecreg}.

Note that in order for $\tilde T(\transp, \sk)$ as given by \eqref{49} to be a tensor field, $\Psi_x$ as given by \eqref{beta} needs to be linear in the $\alpha$ and $Y$ which requires that $\transp^*(x,y) \colon T^*_xM \to T^*_yM$ and $\transp_*(x,y) \colon T_xM \to T_yM$ are linear maps.

It is also clear that if $T \in \tensor^r_s(M)$ is a {\it smooth} type $(r,s)$ tensor field then setting
\begin{equation}
\hat T\gt{a_1}{a_r}{b_1}{b_s}(\transp, \sk)= T\gt{a_1}{a_r}{b_1}{b_s}
\end{equation}
gives an embedding
\begin{align*}
\sigma^r_s &\colon  \tensor^r_s(M) \to \CEM^r_s(M), \\
T & \mapsto \hat T.
\end{align*}

We have seen that by combining a transport operator with a smoothing kernel we may smooth tensor distributions. It is remarkable that all linear and continuous mappings from $\cD'^r_s(M)$ into $\cT^r_s(M)$ are of this form in the following sense: there is an isomorphism of locally convex spaces
\begin{gather*}
 L(\cD'^r_s(M), \cT^r_s(M)) \cong \\
 \Gamma(M \times M, T^r_sM \boxtimes T^s_rM) \otimes_{C^\infty(M \times M)} L(\cD'(M), C^\infty(M)),
\end{gather*}

see \cite{vecreg}. For our purposes, elements of $\Gamma(M \times M, T^r_sM \boxtimes T^s_rM)$ are constructed by taking tensor products of $\transp \in \Gamma(M \times M, TM \boxtimes T^*M)$ as in \eqref{46}.
\section{Generalised Lie derivatives}
In this section we consider the Lie derivative of generalised tensor fields. Before doing so we review the definition of the Lie derivative of a distributional vector field as given by  \cite{marsden} (see also \cite{book}). We begin by looking at the Lie derivative of a distributional vector field $X$. If we let $\theta$ be an arbitrary smooth 1-form  then $X(\theta)$ is a distributional scalar field. We now {\it define} the distributional Lie derivative of $X$ with respect to a smooth vector field $Z$ to be that given by requiring the Leibniz rule for $X(\theta)$ to be satisfied, so that
\begin{equation*}
(\Lie_ZX)(\theta):= \Lie_Z(X(\theta))-X(\Lie_Z\theta)
\qquad \forall \theta \in \Omega^1(M).
\end{equation*}
We now define the distributional Lie derivative of a general distributional tensor field $S$.

\begin{definition}[Lie derivative of tensor fields]
Let $S \in \cD'^r_s(M)$. The Lie derivative of $S$ with respect to the smooth vector field $Z \in \TM(M)$ is the element $\Lie_ZS \in \cD'^r_s(M)$ given by
\begin{align*}
\langle(\Lie_ZS)(\theta^1,&\dots,\theta^r,X_1,\dots,X_s),\omega \rangle
= -\langle S(\theta^1,\dots,\theta^r,X_1,\dots,X_s),\Lie_Z\omega \rangle \\
&-\sum_{i=1}^r\langle S(\theta^1,\dots,\Lie_Z\theta^i,\dots
\theta^r,X_1,\dots,X_s),\omega \rangle \\
&-\sum_{j=1}^s\langle S(\theta^1,\dots,\theta^r,
X_1,\dots,\Lie_ZX_j,\dots,X_s),\omega \rangle \\
\end{align*}
for all $\theta^1,\dots,\theta^r \in \Omega^1(M)$, $X_1,\dots,X_s \in \TM(M)$ and $\omega \in \Omega^n_c(M)$.
\end{definition}
Note that if we regard $S \in \cD'^r_s(M)$ as dual to a type $(s,r)$ tensor density $\Psi$ given by
\[ \Psi=\theta^1\otimes\cdots\otimes\theta^r \otimes X_1\cdots\otimes X_s\otimes\omega \]
then the above formula can be written in the more compact form
\begin{equation*}
\langle \Lie_ZS, \Psi\rangle=-\langle S, \Lie_Z\Psi\rangle.
\end{equation*}

In \cite{paper1} we looked at derivatives of a generalised scalar field. This involved defining the Lie derivative $\Lie^{SK}_X \sk = \Lie_X^{\Omega^n}\sk + \Lie_X^{C^\infty} \sk$ of a smoothing kernel which we obtained by differentiating the action of a 1-parameter group of diffeomorphisms. For derivatives of a generalised tensor field $T(\transp, \sk)$ we will also require the Lie derivative of the transport operator $\transp$. 

In principle we can consider the action of two diffeomorphisms $\mu$ and $\nu$ which act separately on the $x$ and $y$ variables. Thinking of $\transp$ as a sum of terms of the form $V^a(x)\otimes \alpha_b(y)$ we can consider the pullback $\mu^*=(\mu_*)^{-1}$ by taking the inverse of the pushforward action on the vector  $V^a(x)$ and the pullback $\nu^*$ by taking the pullback action on the 1-form $\alpha_b(y)$. This gives us the action 
\[ (\mu^*,\nu^*) \colon \TO(M) \to \TO(M). \]
We can also consider two vector fields $X$ and $Y$ with corresponding flows ${\mathrm{Fl}}^X_t$ and ${\mathrm{Fl}}^Y_t$ acting on the $x$ and $y$ variables. This enables us to define the Lie derivative of $\transp$ by
\[
\dLie_{(X,Y)}\transp=\left.\frac{d}{dt}\right|_{t=0}
\left((\mathrm{Fl}_t^X)^*,(\mathrm{Fl}_t^Y)^*\right)\transp.
\]
Varying the $x$ and $y$ variables separately we have two Lie derivatives $\dLie_{(X,0)}$ and $\dLie_{(0,Y)}$ 
satisfying
\[
\dLie_{(X,Y)}\transp=\dLie_{(X,0)}\transp +\dLie_{(0,Y)}\transp.
\]
We abbreviate $\dLie_X^{TO}\transp \coleq \dLie_{(X,X)} \transp$.

The explicit formulae are given by
\[
(\Lie_{(X,0)}\transp)^a{}_b(x,y)=X^c(x)\frac{\partial \transp^a{}_b(x,y)}{\partial x^c}
-\transp^c{}_b(x,y)\frac{\partial X^a}{\partial x^c}(x)
\]
and
\[
(\Lie_{(0,Y)}\transp^a{}_b)(x,y)=Y^c(y)\frac{\partial \transp^a{}_b(x,y)}{\partial y^c}
+\transp^a{}_c(x,y)\frac{\partial Y^c}{\partial y^b}(y)
\]
so that $(\Lie_{(X,0)}\transp)(x,y)$ corresponds to the Lie derivative of the
transport operator with respect to the vector field $X$ at $x$ keeping
$y$ fixed (i.e. thinking of $\transp(x,y)$ as a vector field at
$x$). Similarly, $(\Lie_{(0,X)}\transp)(x,y)$ corresponds to the Lie derivative
of the transport operator with respect to the vector field $Y$ at $y$
keeping $x$ fixed (i.e. thinking of $\transp(x,y)$ as a covector field
at $y$).

Having calculated the Lie derivative of the transport operator we are
now in a position to look at the Lie derivative of a generalised
tensor field. Given a generalised tensor field $T \in \CEM^r_s(M)$
then for fixed $\sk \in \SK(M)$ and fixed $\transp \in \TO(M)$ we
know that $\tilde T:= T(\transp, \sk)$ is a smooth type $(r,s)$
tensor field and hence we may calculate its (ordinary) Lie derivative
with respect to a smooth vector field $X \in \vect(M)$.

As in the case of the generalised Lie derivative of a scalar field, we
find the correct definition for the Lie derivative of a generalised
tensor field $T$ with respect to a vector field $X$ by differentiating
the pullback of $T$ along the flow of $X$, i.e., $(\Fl^X_t)^* T$, at
time $t=0$. This leads to the following definition:

\begin{definition}[Generalised Lie derivative of tensors] \label{genlietensors}
For $T \in \CEM^r_s(M)$ and $X \in \mathfrak{X}(M)$ we define
\begin{equation} \label{genlietensor}
(\gLie_X T)(\transp,\sk) := \Lie_X(T(\transp, \sk))
- d_1T(\transp, \sk)(\Lie_X^{TO}\transp)
- d_2T(\transp, \sk)(\Lie_X^{SK} \sk).
\end{equation}
\end{definition}

As in the scalar case, $\mathrm{d}_i$ denotes the differential with respect to the $i$th variable in the sense of \cite{KM}.

\begin{remark}
  Formula \eqref{genlietensor} must also be used for {\it scalar} fields which
  have an $\transp$ dependence. Such fields may arise, for example,
  from the contraction of a generalised vector field with a
  1-form. For scalar fields with no $\transp$ dependence the above
  formula reduces to that given in \cite{paper1} for generalised
  scalar fields.
\end{remark}

We now give an explicit formula for the Lie derivative of an embedded vector field $Y$. For notational ease we first consider the special case where $Y$ is continuous so that we do not have to consider distributional derivatives. The embedded vector field is given by
\[
\iota^1_0(Y)(\transp, \sk)^a= \int_{y \in M}  Y^b(y) {\transp}^a{}_b(x,y)\sk_x(y).
\]
Taking the generalised Lie derivative of this according to Definition \ref{genlietensors} gives

\begin{align*}
\gLie_X(\iota^1_0(Y))(\transp,\sk)^a(x) &=
\int_{y \in M}Y^b(y)\bigl((\dLie_{(X,0)}\transp^a{}_b)(x,y)\sk_x(y) \\
& \qquad +
\transp^a{}_b(x,y)(\dLie^{C^\infty}_X\sk)_x(y)\bigr) \\
&\quad -\int_{y \in M}Y^b(y)(\Lie_X^{TO}\transp^a{}_b)(x,y)\sk_x(y) \\
&\quad -\int_{y \in M}Y^b(y)\transp^a{}_b(x,y)(\Lie^{SK}_X\sk)_x(y) \\
&=-\int_{y \in M}Y^b(y)(\dLie_{(0,X)}\transp^a{}_b)(x,y)\sk_x(y) \\
&\quad -\int_{y \in M} Y^a(y)\transp^a{}_b(x,y)(\Lie_X^{\Omega^n}\sk)_x(y) \\
&=\int_{y \in M} (\Lie_XY)^b(y)\transp^a{}_b(x,y)\sk_x(y)\\
&= \iota^1_0(\Lie_XY)(\sk,\transp)^a(x),
\end{align*}
hence $\gLie_X(\iota^1_0(Y)) = \iota^1_0(\dLie_X Y)$.

Turning now to the general case a similar calculation to the above shows that for a distributional tensor field $S \in \cD'^r_s(M)$ we have
$\gLie_X(\iota^r_s(S))=\iota^r_s(\Lie_XS)$.

For a smooth type $(r,s)$ tensor field $S$ we have $\sigma^r_s(S)(\transp, \sk)= S$ and since there is no dependence on the smoothing kernel or
transport operator the generalised Lie derivative is the same as the
ordinary Lie derivative so that $\gLie_Z(\sigma^r_s(S)) = \sigma^r_s(\Lie_ZS)$.

Combining these two results we have
\begin{proposition}\label{prop4}
\leavevmode
\begin{enumerate}[label=(\alph*)]
\item The embedding $\iota^r_s$ of distributional tensor fields 
commutes with the Lie derivative so
that
\[ \gLie_X(\iota^r_s(S))=\iota^r_s(\Lie_X S). \]
\item The embedding $\sigma^r_s$ of smooth tensor fields commutes with the Lie derivative so that
\[ \gLie_X(\sigma^r_s(S))=\sigma^r_s(\Lie_X S). \]
\end{enumerate}
\end{proposition}

\begin{remark}The deeper reason for Proposition \ref{prop4} comes from looking at the induced action of a diffeomorphism $\mu \colon M \to N$ on the space of generalised tensor fields. If $T \in \CEM(N)$ is a generalised tensor field on $N$ then we may pull it back to a generalised tensor field $\mu^*T$ on $M$ by defining 
\[
(\mu^* T)(\transp, \sk)(x) \coleq (D_{\mu(x)}\mu^{-1})^r_s(T((\mu_*, \mu_*)\transp, (\mu_*, \mu_*)\sk)(\mu(x)).
\]
It is readily verified that the action of the diffeomorphism commutes with the embedding so that if $S \in \cD'^r_s(M)$ is a distributional tensor field then
\[
\mu^*(\iota^r_s(S))=\iota^r_s(\mu^*S)
\]
If we now take $\mu$ to be the flow ${\mathrm{Fl}}^X_t$ of a (complete) vector field then we have 
\[
(\mathrm{Fl}^X_t)^*(\iota^r_s(S))=\iota^r_s ((\mathrm{Fl}^X_t)^*S).
\]
Differentiating this with respect to $t$ and using the fact that for any $T \in \CEM^r_s(M)$ we have
\begin{equation} \label{genlie2}
\gLie_XT=\left.\frac{d}{dt}\right|_{t=0}
{(\mathrm{Fl}}^X_t)^*T
\end{equation}
this immediately gives
\[
\gLie_X(\iota^r_s(S))=\iota^r_s(\Lie_X S).
\]
Thus, the fact that the generalised Lie derivative commutes with the embedding follows from the fact that the action of a diffeomorphism commutes with the embedding.
\end{remark}

\section{The quotient construction and the algebra of generalised tensor
  fields}

Having looked at the properties of the basic space $\CEM^r_s(M)$ we turn to the definition of generalised tensor fields $\CG^r_s(M)$. These are defined as moderate tensor fields modulo negligible tensor fields.

Similarly to the nets of smoothing kernels $(\sk_\e)_\e$ of the scalar case, one needs to introduce a suitable asymptotic structure on nets of transport operators $(\transp_\e)_\e$ such that in the limit $\e \to 0$, $\transp_\e \otimes \sk_\e$ converges to the identity in the right way. The respective definitions are as follows:

\begin{definition}[Admissible nets of transport operators]
 A net $(\transp_\e)_\e \in \TO(M)^I$ is called \emph{admissible} if
 \begin{enumerate}[label=(\roman*)]
  \item locally around the diagonal in $M \times M$, $(\transp_\e)_\e$ is uniformly bounded for small $\e$, and
  \item $\transp^a{}_{b,\e}(x,x)=\delta^a_b$ for all $x \in M$.
 \end{enumerate}
 The space of all admissible nets of transport operators is denoted $\Upsilon(M)$.
\end{definition}

Condition (i) means that for each chart $U$ on $M$ and all multiindices $\k,\l$ the derivative $\partial^\k_x \partial_y^\l \Upsilon^a{}_{b,\e}$ is uniformly bounded in a neighborhood of each point $(x,x)$ of the diagonal; condition (ii) simply means that $\Upsilon^*(x,x)$ and $\Upsilon_*(x,x)$ are the identity mappings.

The linear space corresponding to this affine space is introduced as
\[ \Upsilon_0(M) \coleq \{ (\transpp_\e)_\e \in \TO(M)^I\ |\ (\transp_\e)_\e \in \Upsilon(M) \Rightarrow (\transp_\e)_\e + (\transpp_\e)_\e \in \Upsilon(M) \}. \]

The following definitions are the obvious generalisations of the scalar case (we refer to \cite{distcurv,bigone} for detailed proofs in a slightly extended setting):

\begin{definition}[Moderate Tensors] \label{moderate} The tensor field $T \in
\CEM^r_s(M)$ is called moderate if
$\forall K \subset M$ compact $\forall j,k,l \in \Nat_0$ $\forall \transp \in \Upsilon(M)$, $\transp_1, \dotsc, \transp_j \in \Upsilon_0(M)$ $\forall \sk \in \MC(M)$, $\sk_1, \dotsc, \sk_k \in \MC_0(M)$ $\forall X_1, \dotsc, X_l \in \vect(M)$ $\exists N \in \Nat$:
\begin{equation} \label{mod}
\begin{aligned}
&\sup_{x\in K} \|\Lie_{X_1}\dots \Lie_{X_l} (\mathrm{d}_1^j \mathrm{d}_2^k
T(\transp_\e, \sk_\e)\\
&\qquad\qquad(\transp_{1,\e}, \dotsc, \transp_{j,\e}, \sk_{1,\e}, \dotsc, \sk_{k,\e}))(x)\|
= O(\eps^{-N}) \qquad(\eps \to 0)
\end{aligned}
\end{equation}
where $||\cdot||$ denotes the norm induced by some background metric.

The set of moderate tensors in $\CEM^r_s(M)$ is denoted $\CM{}^r_s(M)$.
\end{definition}

Note that the space of moderate tensors does not depend upon the choice of background metric used to define the above norm. The inclusion of the differentials with respect to $\transp$ and $\sk$ makes the definition look quite complicated but for embedded fields the dependence is at worst linear so that in practice there are no significant complications caused by this.

\begin{definition}[Negligible tensors]
The tensor field $\tilde T \in \CM{}^r_s(M)$ is negligible if
$\forall K \subset M$ compact $\forall j,k,l,m \in \Nat_0$ $\forall \transp \in \Upsilon(M)$, $\transp_1, \dotsc, \transp_j \in \Upsilon_0(M)$ $\forall \sk \in \MC(M)$, $\sk_1, \dotsc, \sk_k \in \MC_0(M)$ $\forall X_1, \dotsc, X_l \in \vect(M)$:
\begin{equation} \label{neg}
\begin{aligned}
& \sup_{x\in K} \|\Lie_{X_1}\dots \Lie_{X_l} (\mathrm{d}_1^j \mathrm{d}_2^k
T(\transp_\e, \sk_\e)\\
& \qquad\qquad (\transp_{1,\e}, \dotsc, \transp_{j,\e}, \sk_{1,\e}, \dotsc, \sk_{k,\e}))(x)\|
= O(\eps^{-m}) \qquad(\eps \to 0).
\end{aligned}
\end{equation}
\end{definition}

We should note that the above formulae also apply to type $(0,0)$ tensor
fields (i.e., scalar fields) which depend on  $\transp$. Such fields arise for
example from contraction of higher valence tensors. 
After taking this point
into account it follows from the above definitions that one can test for 
moderateness and negligibility by looking at the scalar field obtained 
by contraction.

\begin{proposition}[Saturation]\label{saturation}
A generalised tensor field $\tilde T \in \CEM^r_s(M)$ is moderate (respectively negligible) iff for all smooth covector fields $\theta^i \in \form$ and smooth vector fields $X_j \in \vect(M)$, the generalised scalar field
\[ F(\transp,\sk)(x)= \tilde T(\transp, \sk)(x)(\theta^1(x) \dots \theta^r(x),X_1(x) \dots X_s(x)) \]
obtained by contraction is moderate (respectively negligible) when regarded as an element of $\CEM^0_0(M)$. 
\end{proposition}

\begin{proposition}\label{submodule}
$\CM{}^r_s(M)$ is a $\CE^0_0(M)$-module with $\CN^r_s(M)$ as a submodule. 
\end{proposition}

\begin{proof}
As with the case of scalar fields the definitions of moderate and negligible may be used to establish the following results which prove
the proposition:
\begin{enumerate}[label=(\alph*)]
\item $\tilde f \in \CM{}^0_0(M)$, $S \in \CM{}^r_s(M)$ $\Rightarrow$ $\tilde f S \in \CM{}^r_s(M)$;
\item $\tilde f \in \CN^0_0(M)$, $S \in \CM{}^r_s(M)$ $\Rightarrow$ $\tilde f S \in \CN^r_s(M)$;
\item $\tilde f \in \CM{}^0_0(M)$, $S \in \CN^r_s(M)$ $\Rightarrow$ $\tilde f S \in \CN^r_s(M)$.\qedhere
\end{enumerate}
\end{proof}

We next examine the properties of the embeddings of distributional and smooth tensor fields into the basic space. As an immediate consequence of the analytical properties of the combination of admissible nets of transport operators with test objects \cite{distcurv} one may establish the following proposition.

\begin{proposition}
\leavevmode
\begin{enumerate}[label=(\alph*)]
\item $\iota^r_s({\cD'}^r_s(M)) \subseteq \CM{}^r_s(M)$.
\item $\sigma^r_s(\mathcal{T}^r_s(M)) \subseteq \CM{}^r_s(M)$.
\item $(\iota^r_s-\sigma^r_s)(\mathcal{T}^r_s(M)) \subseteq \CN^r_s(M)$.
\item If $T\in {\cD'}^r_s(M)$ and $\iota^r_s(T)\in \CN^r_s(M)$, then $T=0$.
\end{enumerate}
\end{proposition}

We may also show that as in the scalar case moderateness and negligibility are stable under the action of the generalised Lie derivative. 

\begin{proposition}
Let $X \in \vect(M)$. Then,
\begin{enumerate}[label=(\alph*)]
\item $\gLie_X(\CM{}^r_s(M)) \subseteq \CM{}^r_s(M)$, and
\item $\gLie_X(\CN^r_s(M)) \subseteq \CN^r_s(M)$.
\end{enumerate}
\end{proposition}
We are now in a position to define the space of generalised tensor fields:

\begin{definition}[Generalised tensor fields] \label{generalisedtensors}
We define the $\CG(M)$-module $\CG^r_s(M)$ of generalised type $(r,s)$ tensor fields by
\begin{equation}
\CG^r_s(M)=\CM{}^r_s(M)/\CN^r_s(M).
\end{equation}
\end{definition}

We now consider tensor operations on $\CG^r_s(M)$.  Let $\tilde S \in \CEM^r_s(M)$ and $\tilde T \in \CEM^t_u(M)$. Since $\tilde S(\transp, \sk)$ and $\tilde T(\transp, \sk)$ are smooth tensor fields we may define $\tilde S \otimes \tilde T \in \CEM^{r+t}_{s+u}(M)$ by
\begin{equation}\label{gamma}
(\tilde S \otimes \tilde T)(\transp, \sk) = (\tilde S(\transp, \sk))\otimes(\tilde T(\transp, \sk)).
\end{equation}
In a similar way as in the scalar case \cite{paper1} one may use the definitions of moderateness and negligibility to show that if both $\tilde T$ and $\tilde S$ are moderate then $\tilde T \otimes \tilde S$ is moderate and that if either $\tilde T$ or $\tilde S$ is negligible then so is $\tilde T \otimes \tilde S$. We may therefore define the tensor product as follows:

\begin{definition}[Tensor product]\label{tensorproduct}
The tensor product of $[\tilde S] \in \CG^r_s(M)$ and $[\tilde T] \in \CG^t_u(M)$ is defined by
\begin{equation}
[\tilde T] \otimes [\tilde S]=[\tilde T \otimes \tilde S]
\label{52}
\end{equation}
where $\tilde T \otimes \tilde S $ is given by equation \eqref{gamma} above.
\end{definition}

In the same way one can show that if $\tilde T$ is obtained from $\tilde S$ by contraction on a pair of indices, then $\tilde T$ is moderate if $\tilde S$ is moderate and that $\tilde T$ is negligible if $\tilde S$ is negligible, hence we may define contraction of generalised tensor field as follows

\begin{definition}[Contraction] \label{contraction}
Let $[\tilde S] \in \CG^r_s(M)$. We may define an element of $\CG^{r-1}_{s-1}(M)$ by making a contraction according to the formula
\[ [\tilde S]^{a\dots e \dots b}_{c \dots e \dots d}= [\tilde S^{a\dots e \dots b}_{c \dots e \dots d}]. \]
\end{definition}

We now define the generalised tensor algebra as
\[
\hCG(M) = \bigoplus_{r,s\in\Nat} \CG{}^r_s(M). 
\]
From Definition \ref{tensorproduct} and Definition \ref{contraction} we see that $\hCG(M)$ is closed under the operations of tensor product and contraction. 

We summarise the properties we have established in the following theorem:

\begin{theorem}
\leavevmode
\begin{enumerate}[label=(\alph*)]
\item The generalised tensor algebra $\hCG(M)$ is an associative differential algebra with product the tensor product $\otimes$, and derivatives given by the generalised Lie derivatives $\hat\Lie_X$ for $X \in \vect(M)$.
\item The algebra is closed under the action of contraction.
\item The space of smooth tensor fields may be embedded by the ``constant map''  $\sigma^r_s$  and the algebra of smooth tensor fields forms a subalgebra of $\hCG(M)$.
\item For each $r,s \in \Nat_0$ there is a linear  map $\iota^r_s$ which embeds ${\cD'}^r_s(M)$ as a $C^\infty(M)$-module of $\CG^r_s(M)$, and the embedding $\iota^r_s$ coincides with $\sigma^r_s$ when restricted to smooth tensor fields.
\item The embeddings $\iota^r_s$ and $\sigma^r_s$ commute with the Lie derivative so that $\hat\Lie_X(\iota^r_s(T))=\iota^r_s(\Lie_X T)$ and $\hat\Lie_X(\sigma^r_s(T))=\sigma^r_s(\Lie_X T)$.
\end{enumerate}
\end{theorem}

We end this section by considering association for tensor fields. The definition of association is much the same as for scalars. 

\begin{definition}[Association]
We say that $[T] \in \CG^r_s(M)$ is associated to 0 (denoted $[T] \approx 0$) if for each $\Psi \in \tCD^s_r(M)$ we have
\[
\lim_{\eps \to 0}\int \tilde
T\gt abcd(\transp_\eps, \sk_\eps)(x)\Psi\gt cdab(x)= 0 
\qquad \forall \sk \in \tilde{\mathcal{A}}(M),\ \forall \transp \in \transp(M).
\]
We say two elements $[S],[T] \in \CG^r_s(M)$ are associated and write $[S] \approx [T]$ if $[S-T] \approx 0$.
\end{definition}

\begin{definition}[Associated distributional tensor field]
 We say $[T] \in \CG^r_s(M)$ admits $S \in \cD'^r_s(M)$ as an associated tensor distribution if for each $\Psi \in \tCD^s_r(M)$ we have
\[
\lim_{\eps \to 0}\int \tilde
T\gt abcd(\transp_\eps, \sk_\eps)\Psi\gt cdab(x)= \langle S,\Psi \rangle
\qquad \forall \sk \in \tilde{\mathcal{A}}(M),\ \forall \transp \in \transp(M).
\]
\end{definition}

Then just as in the scalar case one has the following proposition (with similar proof).

\begin{proposition}
\leavevmode
\begin{enumerate}[label=(\alph*)]
  \item If $S$ is a smooth tensor field in $\mathcal{T}^r_s(M)$ and 
$T \in \cD'^t_u(M)$
then
\begin{equation}
\iota(S)\otimes \iota(T) \approx \iota(S \otimes T). \label{55}
\end{equation}
\item If $S$ and $T$ are  continuous tensor fields then
\begin{equation}
\iota(S)\otimes \iota(T) \approx \iota(S \otimes T). \label{56}
\end{equation}
\end{enumerate}
\end{proposition}

In the scalar case the delts nets of smoothing kernels form a delta-net in the sense that as $\eps \to 0$, $\iota(f)(\sk_\eps) \rightarrow f$ in $\cD'(M)$ for $f \in \cD'(M)$. The following result shows that this remains true in the tensor case.   

\begin{proposition}\label{convergeinD}
Given $S \in \cD'{}^r_s$ then for all $\sk \in \atil$, $\transp \in \transp(M)$ and $\tilde T \in \tilde\cD^s_r(M)$ we have
\[
\lim_{\eps \to 0}\int_{x \in M}\iota{}^r_s(S)(\transp_\e, \sk_\e)(x){\tilde T}(x)=\langle S, \tilde T \rangle,
\]
i.e., $\iota{}^r_s(S)(\transp_\e, \sk_\e) \rightarrow S$ in $\cD'{}^r_s(M)$ as $\eps \to
0$.
\end{proposition}
The following is a trivial corollary which generalises the
corresponding scalar result.
\begin{corollary}
At the level of association the embedding does not depend upon the transport operator or the smoothing kernel in the sense that given different admissible nets of transport operators $\transp$ and $\tilde \transp$ and different delta nets of smoothing kernels $\sk$ and $\tilde \sk$ we have
\[
\lim_{\eps \to 0}\int_{x \in M}\iota{}^r_s(S)(\transp_\e, \sk_\eps){\tilde T}(x)=\lim_{\eps \to 0}\int_{x \in
  M}\iota{}^r_s(S)(\tilde \transp_\e, {\tilde \sk}_\eps)(x){\tilde T}(x)
\]
for $S \in \cD'^r_s(M)$ and $\tilde T \in \tilde{\mathcal{D}}^s_r(M)$.
\end{corollary}
This shows that if one regards our Colombeau type theory as a method for calculating with smoothed distributional tensor fields, then the distributional limit as $\eps \to 0$  exists for embedded distributions and does not depend on the choice of transport operators or smoothing kernels. 

\section{Generalised Differential Geometry and Applications to 
General Relativity}
In the previous section we established the key structural properties of the generalised tensor algebra where we showed that it is closed under the operations of tensor product and contraction and also closed under the action of the generalised Lie derivative. Furthermore, we showed that there exists a canonical embedding of distributional tensor fields given by $\iota^r_s$ and that this embedding commutes with the Lie derivative.  However, from the point of view of applications the key property of generalised tensor fields is that if $T$ is an element of the basic space $\CEM^r_s(M)$ then for any fixed smoothing kernel $\sk$ and any fixed transport operator $\transp$ the tensor field $\tilde T$ given
by
\[ \tilde T:=T(\transp, \sk) \]
is a smooth tensor field, so that we may apply all the usual operations of smooth differential geometry to it. In particular, we can calculate the covariant derivative of a generalised tensor field. Moreover, one can apply the ordinary Lie derivative of smooth tensor fields for fixed $\transp$ and $\sk$ and define $(\tilde\Lie_X T)(\transp,\sk) \coleq \Lie_X ( T(\transp, \sk))$.

We now look at the covariant derivative of a generalised tensor field in the basic space. Let $\nabla$ be a smooth connection and $Z$ a smooth vector field. For $T \in \CEM^r_s(M)$ we define the generalised tensor field $\nabla_Z T$ to be given by 
\[
(\nabla_ZT)(\transp, \sk):= \nabla_Z( T(\transp, \sk )).
\]
Furthermore, if $T$ is moderate then so is $\nabla_Z T$, and if $T$ is negligible then so is $\nabla_Z T$, so that we may define the covariant derivative of a generalised tensor field to be given by
\[
\nabla_Z[T] \coleq [\nabla_Z T].
\]

\begin{lemma}\label{lemmaa}
Let $S \in \cD'{}^r_s(M)$ be a distributional type $(r,s)$ tensor field, $T \in \tensor^s_r(M)$ a smooth type $(s,r)$ tensor field and $Z \in \vect$ a smooth vector field. Then
\[
\tilde\Lie_Z (\iota{}^r_s(S)^a T_a) \approx \iota{}^0_0(\Lie_Z (S^a T_a)).
\]
\end{lemma}

\begin{proof}
We will illustrate this by considering a distributional vector field $X \in \cD'{}^1_0(M)$ and contracting with $\theta \in \Omega^1(M)$. 
Let $\mu$ be a smooth density of compact support; then by Proposition \ref{convergeinD} we have 
\begin{gather*}
 \lim_{\e \to 0} \int \tilde \Lie_Z ( \iota^1_0(X)^a \theta_a)(x)\mu(x) \\
= - \lim_{\e \to 0} \int \iota^1_0(X)^a(x) \theta_a(x) (\Lie_Z \mu)(x) \\
= - \langle X^a \theta_a, \Lie_Z \mu \rangle = \langle \Lie_Z ( X^a \theta_a), \mu \rangle.
\end{gather*}
On the other hand,
\begin{gather*}
 \lim_{\e \to 0} \int \iota^0_0(\Lie_Z(X^a \theta_a))(x)\mu(x) = \langle \Lie_Z ( X^a \theta_a), \mu \rangle
\end{gather*}
so that
\[
\tilde \Lie_Z(\iota^1_0(X)^a \theta_a) \approx \iota^0_0(\Lie_Z ( X^a \theta_a) ).
\]
The general case is seen similarly.
\end{proof}

\begin{proposition}\label{prop23}
Let $S \in \tensor^r_s(M)$ be a $C^1$ type $(r,s)$ tensor field, $Z \in \vect(M)$ a smooth vector field and $\nabla$ a smooth covariant derivative. Then
\begin{equation}\label{delta}
\nabla_Z (\iota{}^r_s(S)) \approx \iota{}^r_s(\nabla_ZS).
\end{equation}
\end{proposition}

\begin{proof}
This follows directly from continuity of $\nabla_Z \colon \cD'^r_s(M) \to \cD'^r_s(M)$.
\end{proof}

\begin{remark}
\leavevmode
 \begin{enumerate}[label=(\alph*)]
  \item For a smooth tensor field $S \in \tensor^r_s(M)$ equation \eqref{delta} is true with equality rather than association.
  \item With a suitable definition of distributional covariant derivative (cf.~\cite[Section 3.1]{book}) Proposition \ref{prop23} is true for $S \in \cD'{}^r_s(M)$. 
  \item Give any given coordinate system $x^\mu$ we can define a covariant derivative which is nothing but the partial derivative in these coordinates. Hence in any given coordinate system the above result is also true if we replace $\nabla_Z$ by the partial derivatives $\partial_\mu$.
  \item The above result is also true for all $S \in \cD'{}^r_s(M)$ if we replace $\nabla_Z$ by $\Lie_Z$.
 \end{enumerate}
\end{remark}

Up to now we have discussed the covariant derivative of a generalised tensor field with respect to a smooth classical connection $\nabla$. We now define a generalised version of this. We may do this by writing a generalised covariant derivative as being given by a (smooth) covariant derivative $\delo$ with respect to some background connection $\gamma$ together with a correction term given by a {\it generalised} type $(1,2)$ tensor field $\hat\Gamma^a_{bc}$. Thus, the generalised covariant derivative of a generalised vector field $X$ is given by
\begin{equation}
(\nabla_Z X)^a
=(\delo_Z X)^a+\hat\Gamma^a_{bc}Z^b X^c.
\label{gencovderiv}
\end{equation}
Note that this does not depend on the choice of background connection if we change the tensorial correction term by the difference of the connection coefficients of the background connections. This leads to the following definition:

\begin{definition}[Covariant derivative]
Let $[T]\gt{a_1}{a_r}{b_1}{b_s} \in \CG^r_s(M)$, let $[\hat\Gamma]^a_{bc} \in \CG^1_2(M)$ and let $Z$ be a smooth vector field. Then we may define $\nabla_ZT \in \CG^r_s(M)$ by
\[ \nabla_Z[T]=[\nabla_ZT] \]
where
\begin{equation} \label{71}
\begin{aligned}
(\nabla_ZT)\gt{a_1}{a_r}{b_1}{b_s} &=
\delo_ZT\gt{a_1}{a_r}{b_1}{b_s}+
Z^c(\hat\Gamma^{a_1}_{dc}T\gt{da_2}{a_r}{b_1}{b_s}+\dots
\\ & \quad +\hat\Gamma^{a_r}_{dc}T\gt{a_1}{a_{r-1}d}{b_1}{b_s}
-\hat\Gamma^{d}_{b_1c}T\gt{a_1}{a_r}{db_2}{b_s}-\dots
-\hat\Gamma^{d}_{b_sc}T\gt{a_1}{a_r}{b_1}{b_{s-1}d})
\end{aligned}
\end{equation}
We note that the above definition also makes sense if we replace $Z$ by a generalised vector field $Z$.
\end{definition}

We now turn to the definition of a generalised metric.  Generalised metrics have been considered in the context of the special algebra of tensor fields by \cite{gen-geom}. There, a number of equivalent definitions of a generalised metric are given. We will use the following definition:
\begin{definition}[Generalised metric] 
We say $g_{ab} \in
\CG^0_2(M)$ is a generalised metric
if
\begin{enumerate}[label=(\roman*)]
 \item $g_{ab} = g_{ab}$, i.e., $g$ is symmetric, and
 \item the map $X^a \mapsto X^a g_{ab}$ from $\CG^1_0(M)$ into $\CG^0_1(M)$ is bijective.
\end{enumerate}
\end{definition}

\begin{proposition}
If $g_{ab}$ is a $C^0$ metric then $\tilde g_{ab}=\iota^0_2(g_{ab})$ is a generalised metric.
\end{proposition}

\begin{proof}
This follows from the fact that if $g_{ab}$ is continuous then $\tilde g_{ab}(\transp_\e, \sk_\e)$ converges uniformly to $g_{ab}$ on compact subsets, which allows one to define the inverse metric via the cofactor formula along the lines of \cite{distcurv}.
\end{proof}

\begin{definition}[Generalised Levi-Civita connection] 
Given a generalised metric $g_{ab}$ one may calculate the generalised Levi-Civita connection by defining $\hat \Gamma^a_{bc}$ according to \begin{equation}
\hat\Gamma^a_{bc}=\frac{1}{2} g^{ad}( g_{bd|c}+
g_{cd|b}- g_{bc|d}) \label{80}
\end{equation}
where $g^{ab}$ is defined by $g^{ad}g_{db}=\delta^a_b$ and $g_{bd|c}$ denotes the covariant derivative of $g_{bd}$ with respect to the background connection $\gamma$.
\end{definition}

One now defines the corresponding generalised covariant derivative according to (\ref{71}) using $\hat\Gamma$ defined in equation (\ref{80}) above.
\begin{proposition}
If $g_{ab}$ is a $C^1$ metric then
\[ \hat \Gamma^a_{bc} \approx \iota[(\Gamma^a_{bc}-\gamma^a_{bc}] \]
where $\hat\Gamma^a_{bc}$ is the generalised Levi-Civita connection of the generalised metric $g_{ab}$, $\hat\Gamma^a_{bc}$ is the Levi-Civita connection of $g_{ab}$ and $\gamma^a_{bc}$ are the connection coefficients of the background connection $\gamma$.
\end{proposition}

\begin{proof}
The proof follows from the fact that for a $C^1$ metric
\begin{equation}
\iota[g^{ad}]\iota[g_{bd}]_{|c} \approx \iota[g^{ad}g_{bd|c}].
\end{equation}
\end{proof}

We next consider the generalised curvature of a generalised connection.

\begin{definition}[Generalised curvature]
Let $\hat\nabla$ be a generalised connection. We may define a type $(1,3)$ generalised curvature tensor $\hat {R^a}_{bcd} \in \CG^1_3(M)$ by
\[
(\hat\nabla_X\hat\nabla_Y-\hat\nabla_Y\hat\nabla_X-
\hat\nabla_{[X,Y]})Z=\hat R(X,Y)Z
\]
where $X$, $Y$ and $Z$ are smooth vector fields. 
\end{definition}

\begin{proposition}
Let $\Gamma^a_{bc}$ define a differentiable connection $\nabla$ and let $\hat\Gamma^a_{bc}$, given by $\hat\Gamma^a_{bc}=\iota[(\Gamma^a_{bc}- \gamma^a_{bc})]$,  be used in equation (\ref{gencovderiv}) to define the generalised connection $\hat \nabla$. Then,
 \[ \hat {R^a}_{bcd} \approx \iota[{ R^a}_{bcd}], \]
i.e., the generalised curvature of the embedded connection $\hat \nabla$ is associated to the embedding of the curvature of $\nabla$.
\end{proposition}

Combining this with our earlier result on connections we have the following result.

\begin{proposition}
If $g_{ab}$ is a $C^2$ metric then
\[ \tilde {R^a}_{bcd} \approx \iota[{ R^a}_{bcd}] \]
where $\tilde {R^a}_{bcd}$ is the generalised curvature of the generalised Levi-Civita connection of $\tilde g_{ab}$ and $R^a{}_{bcd}]$ is the curvature of the standard Levi-Civita connection of $g_{ab}$.
\end{proposition}

By contraction we may define $\tilde R_{bd}=\tilde {R^a}_{bad}$, $\tilde R=\tilde g^{bd} \tilde R_{bd}$ and $\tilde G_{ab}=\tilde R_{ab}-\half\tilde g_{ab}\tilde R$. Then the above result gives the following proposition:

\begin{proposition}
If $g_{ab}$ is a $C^2$-solution of the vacuum Einstein equations $G_{ab}=0$ then 
\[ \tilde G_{ab} \approx 0. \]
\end{proposition}

Thus, if we have have a $C^2$-solution of the vacuum Einstein equations then the embedded metric $\tilde g_{ab}$ also satisfies the Einstein equations at the level of association (although the Bianchi identities hold at the level of equality). The important thing to note about this is that it suggests that for generalised metrics the appropriate version of the Einstein equations is
\[ \tilde G_{ab} \approx 8\pi \tilde T_{ab} \]
where $\tilde T_{ab}$ is the embedding of some distributional energy-momentum tensor. This is in the spirit of the `coupled calculus' approach of \cite{col1} where one performs the algebraic operations and derivatives in the differential algebra $\hCG(M)$, but solves the differential equations at the level of association.

We now consider the case where  $g_{ab}$ is not $C^2$ but satisfies the weaker regularity conditions of Geroch and Traschen \cite{GT} which guarantee the existence of a distributional curvature $R^a_{bcd}$. We show that with some additional technical conditions that guarantee that $\iota^0_2(g_{ab})$ is indeed a generalised metric, $\tilde G_{ab}$ is associated to the embedding of the distributional energy tensor defined by $R^a_{bcd}$. 

\begin{definition}[Geroch Traschen regularity]
A symmetric tensor $g_{ab}$ is called a {\em gt-regular metric} if it is a metric almost everywhere and $g_{ab}$ and $g^{ab}$ are in $\Ll^\infty\cap\Hl^1$.
\end{definition}

In the above definition $\Ll^\infty$ denotes the space of locally bounded functions and $\Hl^1$ denotes the Sobolev space of functions which are locally square integrable and also have locally square integrable first (weak) derivative.  Note that although the above definition appears to be stronger than that in \cite{GT} it is actually equivalent to the original one (see \cite{SVgt} for details). The fact that a gt-regular metric is only defined almost everywhere causes some difficulties. In \cite{SVgt} a class of nondegenerate and stable gt-regular metrics was introduced and it was shown that if these are smoothed componentwise by a suitable class of mollifiers then the curvature of the smoothing $\tilde g^\eps_{ab}$ tends to the (distributional) curvature of $g_{ab}$ in $\cD'$. Rather than go into the complications of defining a nondegenerate and stable metric in the present context we will instead follow \cite{GT} (especially Theorem 4) and work with the slightly larger class of continuous gt-regular metrics. One can then show that given a continuous gt-regular metric we can either derive the (distributional) Riemann curvature $\Riem[g]$ of the gt-regular metric $g_{ab}$ or embed $g_{ab}$ in the algebra to obtain the generalised metric $\tilde g_{ab}$. If we then derive its curvature $\Riem[\tilde g]$ within the generalised setting we find that it is associated with the distributional curvature $\Riem[g]$. This is depicted in the following diagram
\[
\begin{CD}
\Ll^\infty\cap\Hl^1
\ni g_{ab} @>\iota^0_2>>[\tilde g_{ab}]\in\CG^0_2(M)\\
@V\mbox{$\cD'$}VV @VV\mbox{Colombeau}V\\
\Riem[g]@<\approx<<\Riem[\tilde g]
\end{CD}
\] 
More precisely, we have the following
theorem.
\begin{theorem}[Compatibility for the Riemann curvature]\label{gtthm}
Let $g_{ab}$ be a continuous gt-regular metric with Riemann tensor $\Riem[g]$. Let $\tilde g_{ab}:=\iota^0_2(g_{ab})$ be the generalised metric obtained by embedding in the algebra. Then
\[ \Riem[\iota^0_2(g)] \approx \Riem[g]. \]
\end{theorem}

\begin{proof} 
Since $g_{ab}$ is continuous $\iota^0_2(g_{ab})$ defines a generalised metric. We may then obtain the estimates used in deriving the corresponding result in \cite{SVgt} by working with the local form of the smoothing kernel and the transport operators together with the fact that $\transp(x,x)={\rm id}$.
\end{proof}

The following corollary is immediate.
\begin{proposition}\label{gtregular} 
If $g_{ab}$ is a continuous gt-regular metric that satisfies the vacuum Einstein equations then $\tilde g_{ab}=\iota^0_2(g_{ab})$ is a generalised metric which satisfies 
\[ \tilde G_{ab} \approx 0. \]
\end{proposition}

Moving beyond the class of gt-regular metrics it is of considerable interest to find the weakest conditions on $g_{ab}$ which guarantee that $\tilde G_{ab}$ is associated to a (conventional) distribution, so that the source admits a distributional interpretation. We know from the example of conical singularities \cite{cvw,VW1} that it is possible to have metrics which do not satisfy the Geroch and Traschen regularity conditions, but all the same have a distributional energy-momentum tensor. We briefly review this work in the context of the present manifestly coordinate invariant theory.

In \cite{cvw} it was shown that if one computed the scalar curvature density of a cone in $\Real^2$ in Cartesian coordinates it was associated to a delta distribution $\delta^{(2)}(x,y)$ with a numerical factor that depended on the deficit angle. In a subsequent paper (see \cite{VW1}) it was furthermore shown that if one transforms the metric to a new coordinate system the generalised scalar curvature density is associated to the transformed delta distribution. 

In the present paper we have shown that one can embed $g_{ab}$ into the Colombeau algebra $\hCG(M)$ in a manifestly coordinate invariant way.
We now show that, for the case of a 2-dimensional cone, the scalar curvature is associated to a delta distribution.
We outline the calculation below.

In Cartesian coordinates the metric of the two dimensional cone with deficit angle $2(1-A)\pi$ may be written as
\begin{align*}
   g_{ab} & = \tfrac12 (1+A^2) \delta_{ab} + \tfrac12 (1-A^2) m_{ab} \\ 
    m_{ab} & = \begin{pmatrix} {x{}^2-y{}^2\over x{}^2+y{}^2} & {2xy\over
   x{}^2+y{}^2} \\
   {2xy\over x{}^2+y{}^2} & -{x{}^2-y{}^2\over x{}^2+y{}^2}\\
   \end{pmatrix}
\end{align*}
Since $\delta_{ab}$ is already smooth and $A$ is a constant, the only term we need to smooth for embedding the metric into $\check\CG(M)$ is $m_{ab}$.

To show that the scalar curvature $\widetilde R_\e$ of $\tilde g_{ab,\e} = \tilde g(\transp_\e, \sk_\e)$ converges in the sense needed for association, one writes the pairing with a smooth 2-form of compact support in local coordinates as
\begin{gather*}
\int \tilde R_\e \omega(x) \sqrt{\abso{\tilde g_\e(x)}} \,dx \\
= \omega(0,0) \int \tilde R_\e(x) \sqrt{\abso{\tilde g_\e(x)}}\, dx + \int \int_0^1 \tilde R_\e(x) (D\omega)(tx)x \sqrt{\abso{\tilde g_\e(x)}}\,dt \,dx.
\end{gather*}
While the first integral on the right-hand side can easily be evaluated using the Gauss-Bonnet theorem to give the desired result, we need precise estimates for the components $\tilde g_{ab,\e}$ of the regularized metric to show that the second integral vanishes for $\e \to 0$. For this one looks at the integrand inside and outside a neighborhood of zero whose diameter is propertional to $\e$, say $\e R_0$. In the inside one can directly employ homogeneity of the components of the metric and the $L^1$-conditions on $(\sk_\e)_\e$ to obtain the needed estimate. For the outside, one has to find an expression for the constant $C$ appearing in the estimate (away from the origin)
\[ \abso{\partial^\alpha \tilde g_{ab,\e}(x) - \partial^\alpha g_{ab}(x)} \le C \e^q \]
in terms of derivatives of $g_{ab}$, which is again combined with homogeneity of the metric to obtain
\begin{equation} \tR_\eps = \begin{cases}
   O(1/\eps^2) & \text{if $r<\eps R_0$} \\
   O(\eps/r^3) & \text{if $r>\eps R_0$.}\\ \end{cases}
\end{equation}

With this one obtains that given any smooth 2-form $\mu$ of compact support one has
\begin{equation}
\lim_{\eps \to 0} \int \tilde R_\eps \mu = 4\pi(1-A)\langle
\delta^{(2)}, \mu \rangle
\end{equation}
which shows that the generalised scalar curvature is associated to a delta distribution. A similar calculation (but requiring more delicate estimates) can be carried out for the Ricci curvature of a 4-dimensional cone along the lines of those in Wilson \cite{wilson} which gives the following result.  

\begin{proposition}
Let $g_{ab}$ be the conical
metric given in standard cylindrical polar coordinates by
$$
ds^2=dt^2-dr^2-A^2r^2d\phi^2-dz^2
$$
then 
\begin{equation}
\tilde G_{ab} \approx 8\pi \tilde T_{ab}
\end{equation}
where $\tilde T_{ab}$ is the embedding into the Colombeau algebra of
the energy momentum tensor of a cosmic string with delta-function
terms with singular support on the string and with the stress equal to
the density $\mu=2\pi(1-A)$.
\end{proposition}

{\bfseries Acknowledgments.} E. Nigsch was supported by the Austrian Science Fund (FWF) grants P26859 and P30233.

\label{lastpage}
\end{document}